 \documentclass[11pt,letterpaper]{amsart}
 \usepackage{amsmath,amsthm,epsfig,amssymb}
\usepackage{amsthm,amsfonts,amssymb,amscd}
\usepackage[utf8x]{inputenc}
\usepackage{float}
\usepackage{textcomp}
\usepackage{tikz}
\usetikzlibrary{positioning, matrix, arrows}
\usepackage{tikz-cd}
\usepackage[all]{xy}
\usepackage[hyphens]{url}
\usepackage[colorlinks=true,citecolor=blue,urlcolor=blue,linkcolor=blue]{hyperref}
\usepackage{array}
\usepackage{multirow}
\usepackage{enumerate}
\usepackage{graphicx}
\usepackage{mathrsfs}
\usepackage{afterpage}
\usepackage{lmodern}
\usepackage{mathtools}
\usepackage{pifont}

\newtheorem{theorem}{Theorem}[section]
\newtheorem{conjecture}[subsection]{Conjecture}
\newtheorem{lemma}[theorem]{Lemma}
\newtheorem{proposition}[theorem]{Proposition}
\theoremstyle{definition}
\newtheorem{remark}[theorem]{Remark}
\numberwithin{equation}{section}
\newtheorem{definition}[theorem]{Definition}

\newtheorem{proposition-definition}{Proposition-Definition}

\newcommand{\Id}{\text{Id}}

\title[ On a Conjecture of Drinfeld]{On a Conjecture of Drinfeld}

\date{}

\begin{document}
\author{SARBESWAR PAL}  
\email{sarbeswar11@gmail.com, spal@iisertvm.ac.in}
\address{IISER Thiruvananthapuram, Maruthamala P. O., Kerala 695551}
\keywords{rational curve, wobbly bundles, Fano manifold} 
\subjclass[2010]{14J60}

\begin{abstract}
 Let $C$ be a smooth irreducible irreducible projective curve of genus $g \ge 2$. 
Let $\mathcal{M}_C(n, \delta)$ be the moduli space of semi-stable vector bundles on $C$ of rank $n$ and fixed determinant $\delta$ of degree $d$. Then the locus of wobbly bundles is known to be closed in $\mathcal{M}_C(n, \delta)$. It was announced by Laumon and attributed to Drinfeld that  the wobbly locus is pure of co-dimension one,  i.e., they form a divisor in $\mathcal{M}_C(n, \delta)$. This is now known as Drinfeld's conjecture.
In this article, we will give a proof of the conjecture when $n$ and $d$ are coprime. 
\end{abstract}

\maketitle

\section{Introduction}
  Let $C$ be a smooth irreducible projective complex curve of genus $g\geq 2$ and let $K_C$ be its canonical line bundle. Fix a line bundle $\delta$ of degree $d$ on $C$ and let $\mathcal{M}_C(n, \delta)$ be the coarse moduli space parameterizing the semi-stable vector bundles of rank $n$ and fixed determinant $\delta$ over $C$. Recall, a vector bundle $E$ is called very stable if $E$ has no nonzero nilpotent Higgs field $\phi \in H^0(C, \text{Ad}(E) \otimes K_C)$. When $g \geq 2$, G. Laumon proved that a very stable vector bundle is stable and that the locus of very stable bundles is a non-empty open subset of $\mathcal{M}_C(n,\delta)$ \cite[Proposition 3.5]{L}. A  vector bundle is called wobbly if it is not very stable. Hence, the locus of wobbly bundles is a closed subset 
$\mathcal{W}_C  \subset \mathcal{M}_C(n, \delta).$

\begin{conjecture}\label{NC1}
$($Drinfeld, \cite[Remarque 3.6 (ii)]{L}$)$ $\mathcal{W}_C$ is of pure codimension one. 
 \end{conjecture}

The aim of this article is to give a proof of the above conjecture when the rank $n$ and degree $d$ are co-prime. The term ``wobbly'' was introduced in the paper \cite{DP}. The conjecture has been proved completely for rank two by the author of this article and jointly with C. Pauly in \cite{P} and \cite{PP}. Our approach in this article is different from \cite{P} and \cite{PP}.

The stable locus $\mathcal{M}^s_C(n, \delta) \subseteq \mathcal{M}_C(n, \delta)$ is smooth and is a Fano manifold of Picard number one \cite[Theorem B, Theorem F]{DN}. Hence, it can be covered by rational curves \cite[Ch V, Theorem 1.6.1]{Kol}. In particular, it is uniruled. In other words, there is a dominant rational map $\Phi: Y \times \mathbb{P}^1 \cdots \to \mathcal{M}_C^s(n, \delta)$, where $Y$ is a variety of dimension $\text{dim}(\mathcal{M}_C^s(n, \delta))-1$. Then by \cite[Corollary 2.13]{D}, a general rational curve in $\mathcal{M}^s_C(n, \delta)$ is free (for the definition of free rational curves, see Section \eqref{S2}).

 Our main idea to prove the conjecture is to study the nonfree rational curves in $\mathcal{M}_C^s(n, \delta)$. In fact, we show that a stable vector bundle on $C$ is wobbly if and only if there is a nonfree rational curve passing through it. Let $\text{Hom}^{\text{nonfree}}(\mathbb{P}^1, \mathcal{M}^s_C(n, \delta))$ be the scheme parametrizing non-free rational curves in $\mathcal{M}_C^s(n, \delta)$. Then it is enough to show that the subvariety spanned (see \ref{MAR31D}) by $\text{Hom}^{\text{nonfree}}(\mathbb{P}^1, \mathcal{M}^s_C(n, \delta))$ is pure of co-dimension one. For this, we show that every irreducible component of $\text{Hom}^{\text{nonfree}}(\mathbb{P}^1, \mathcal{M}^s_C(n, \delta))$ contains a rational curve $f$ such that $f^*T(\mathcal{M}_C^s(n, \delta)$ has exactly one direct summand of negative degree ( in fact, of degree $-1)$. Then, one can easily show that the subvariety spanned by such a component is irreducible of co-dimension one.   

{\bf Organization of the paper}: In Section \ref{S2}, we recall some basic facts that we need in the subsequent sections. In Section \ref{S3}, we give a necessary and sufficient criterion for a stable vector bundle being wobbly. More precisely, we prove the following theorem:

 \begin{theorem}\label{IT2}
A  vector bundle $E \in \mathcal{M}_C^s(n, \delta)$  is wobbly if and only if there is a nonfree rational curve passing through it provided $E$ satisfied the condition $($\ref{assumption}$)$ in Section \ref{S3}.
\end{theorem}

 In Section \ref{S4}, we will study the subvariety spanned by the nonfree rational curves in $\mathcal{M}_C(n, \delta)$. In fact, we will show that this subvariety forms a divisor. Finally, using Theorem \eqref{IT2}
we, prove the following Theorem:
\begin{theorem}\label{T31}
If $(n, d)=1$, then $\mathcal{W}_C$ is of pure codimension one.
\end{theorem}
This settles the Conjecture \eqref{NC1} in the fine moduli case. Our method also works for the coarse moduli case with some mild hypothesis, which may be known to be true to experts (see remark \ref{R31}).

\section{Preliminaries}\label{S2}

In this section we recall some basic definitions and facts about Fano manifolds which can be found in  \cite{Kol}. Let $X$ be a Fano manifold of dimension $n$. A parametrized rational curve in $X$ is a morphism $\mathbb{P}^1 \to X$ which is birational over its image. We will not distinguish parametrized rational curves from its image $f(\mathbb{P}^1)$ and we call $f$, a rational curve in $X$. Recall that any vector bundle on $\mathbb{P}^1$ is a direct sum of line bundles. Thus, given a rational curve $f: \mathbb{P}^1 \to X$, the pull back of the tangent bundle $TX$ can be written as 
\begin{equation}\label{e1}
f^*TX = \mathcal{O}_{\mathbb{P}^1}(a_1) \oplus \mathcal{O}_{\mathbb{P}^1}(a_2) \oplus  \cdots \oplus \mathcal{O}_{\mathbb{P}^1}(a_n) \,\, \text{with} \,\, a_1 \ge a_2\ge \cdots \ge a_n.
\end{equation}
The rational curve $f$ is said to be {\it free} if all the integers $a_1, a_2, \cdots, a_n$ in the above equation are non-negative. We say a rational curve is {\it nonfree} if it is not free.

From now on, we will assume that $X$ is a Fano manifold of dimension $n$ with Picard number $1$. It is known that through every point of $X$, there is a rational curve in $X$. The degree of $f^*K_X^{-1}$ is called the {\it anti-canonical degree} of the curve $f: \mathbb{P}^1 \to X$.

Following the conventions of \cite{Kol}, let $\text{Hom}(\mathbb{P}^1, X)$ be the quasi projective scheme parametrizing the morphisms $\mathbb{P}^1 \to X$ \cite[Theorem 1.10]{Kol}. There is a morphism 
\begin{equation}\label{MAR31}
F: \mathbb{P}^1 \times \text{Hom}(\mathbb{P}^1, X) \to X
\end{equation}
called the {\it universal morphism} or the {\it evaluation map} \cite[P.114 3.3]{Kol}.

\begin{definition}\label{MAR31D}
We say a subvariety $Z \subset X$ is {\it spanned} by a closed subset $\mathcal{K} \subset \text{Hom}(\mathbb{P}^1, X)$ if $\overline{F(\mathbb{P}^1 \times \mathcal{K})} = Z$. 
 \end{definition}

 Let $f: \mathbb{P}^1 \to X$ be a rational curve and let 
\begin{equation}\label{PE1}
\phi(p, f): H^0(\mathbb{P}^1, f^*TX) \to T_f(p)X
\end{equation}
be the natural evaluation map. \\
Let 
\[
F: \mathbb{P}^1 \times \text{Hom}(\mathbb{P}^1, X) \to X
\]
be the evaluation map . Then we have  
\begin{proposition}\label{PP1}
   $dF(p, f)= df(p) + \phi(p, f)$  
\end{proposition}
\begin{proof}
 See \cite[Ch. II, Proposition 3.4]{Kol}.    
\end{proof}
\begin{lemma}[Izzet Coskun and Eric Riedl]\label{July7}
Let $V$ be an irreducible component of the closed subset of $\text{Hom}(\mathbb{P}^1, X)$ consisting non-free rational curves. Further assume that  for a general element $f \in V$, $f^*TX= \oplus_{i=1}^{n-1}\mathcal{O}(a_i) \oplus \mathcal{O}(-1)$, where $a_i \ge 0$ for all $i$. Then 
$\overline{F(\mathbb{P}^1 \times V)}$ has co-dimension one. 
\end{lemma}
\begin{proof}
    The proof of the above Lemma suggested by Izzet Coskun and Eric Riedl in a private communication. I include it here. 

   Let $Z'$ be the image of the evaluation morphism and let $Z$ be a desingularization. A general map in $V$, will lift to $Z$.\\
  Consider a very general $f: \mathbb{P}^1 \to Z \to X$  in $V$; we denote the first map by $g$ and the second map ny $h$. Now consider the normal bundle exact sequence: 
  \[
  0 \to g^* TZ  \to f^* TX \to f^* N_h \to 0,
\]
where $N_h$ denotes the normal bundle.
   
  By assumption, $H^1(f^* TX) =0$, so $f$ is a smooth point of  $\text{Hom}(\mathbb{P}^1, X)$. Since $g$ passes through a very general point of $Z$, the morphism $g$ is free. Hence, $g$ is a smooth point of $\text{Hom}(\mathbb{P}^1, Z)$. Since the deformations of $f$ lie in $Z$ and cover $Z$, we have $\text{Hom}(\mathbb{P}^1, Z) =  \text{Home} (\mathbb{P}^1, X)$ around $f$ and $g$. Hence, the tangent spaces have to be equal, so we have that $H^0(g^* TZ) = H^0(f^* TX)$. By the classification of bundles on $P^1$, we must have 
\[
f^*TX= \oplus_{i=1}^{n-1}\mathcal{O}(a_i) \oplus \mathcal{O}(-1)
\]
and 
\[
g^* TZ = \oplus_{i=1}^{n-1}\mathcal{O}(a_i)
\]
and $N_h = O(-1)$,
so that $Z$ is of codimension 1. 
This is because all the sumands in $g^* TZ$ have to be nonnegative and less than or equal to the summands in $f^* TX$. \\
On the other hand, if any of the summands were actually smaller, then the their global sections could not be equal. Hence, $g^* TZ$ must consist of all the nonnegative summands of $f^*TX$.
Hence, the dimension of $Z$ must be $n-1$.

\end{proof}

\subsection{The Hitchin map}\label{SS1}

Let $\mathcal{M}_H(n, \delta)$ be the moduli space of Higgs bundles of rank $n$ with fixed determinant $\delta$ on $C$. It is known that the cotangent bundle $T^*\mathcal{M}_C^s(n, \delta) \subset \mathcal{M}_H(n, \delta)$ as an open dense subset. Given a Higgs bundle $(E, \varphi)$, where 
$\varphi :E \to E \otimes K_C$ is a trace free homomorphism of vector
bundles, one can define  its characteristic coefficients $a_i \in \Gamma(X,K_C^i)$
for $2 \leq i \leq n$ by setting $a_i={(-1)}^i\text{Trace}
\wedge^i\varphi$, which induces a morphism

\begin{equation}\label{RE98}
h: \mathcal{M}_H(n, \delta) \to W:= \oplus_{i=2}^nH^0(C, K_C^i),
\end{equation}
such that $h(E, \varphi)= (a_2, a_3, ..., a_n)$. This morphism is known as the Hitchin map. We also denote the restriction of the Hitchin map $h$ to $T^*\mathcal{M}_C^s(n, \delta)$ by $h$.  Let $N$ be the dimension of $\mathcal{M}_C(n, \delta)$, then the dimension of $W$ is also $N$.  It is known that, $h: T^*\mathcal{M}_C(n, \delta) \to W$ is an algebraically completely integrable Hamiltonian system \cite[Proposition 4.1. 4.4, P-99-5.1]{NH}. Hence, every component of each fiber of the Hitchin map is a Lagrangian, in particular it has dimension $N$.

\section{Codimension of the wobbly locus}\label{S2.2}
In this section we will give a short proof that the wobbly locus has co-dimension one. However, we can not conclude that the wobbly locus is pure of co-dimension one.

By algebraically completely integrablity, of $T^*\mathcal{M}_C^s(n, \delta)$, every component of $h^{-1}(0)$ is Lagrangian and hence has dimension $N$. Let $X_1, X_2, ..., X_m$ be its components, with $X_1$ be the component corresponding to the zero section of  $T^*\mathcal{M}_C^s(n, \delta)$. There is a forgetful map 
 \[
 F: X_1 \cup X_2 \cup...\cup X_m \to \mathcal{M}_C^s(n, \delta).
 \]
 Let $\mathcal{W}_i$ be the image of $X_i, i= 2, ..., m$. Then the wobbly locus is $\cup_{i=2}^m\mathcal{W}_i$.
 \begin{lemma}\label{Mar28}
  The locus of wobbly bundles is non-empty.
  \end{lemma}
  \begin{proof}
  To see this, let us consider two general stable vector bundles $E_1, E_2$ of rank $r_1, r_2$ and degree $d_1, d_2$ respectively, with $\text{det}(E_1) \otimes \text{det}(E_2) \simeq \delta$ and $r_1 + r_2=n$, such that:\\
 (A) $\mu(E_1) < \mu(E_2)$,\\
 (B) $r_2(r_2-r_1)(g-1) +r_2(r_1(g-1)+d_1)-r_1d_2 > 0$.\\
 Then condition (A) ensures that a general extension 
 \[
 0 \to E_1 \to E \overset{f}{\to}  E_2 \to 0
 \]
is stable \cite[Lemma 2.3]{TB}. On the other hand, condition (B) ensures that $h^0(E_2^* \otimes E_1 \otimes K_C) \ne 0$, \cite[Theorem 0.3]{RTLC}. Let $s$ be a non-zero section of $E_2^* \otimes E_1 \otimes K_C$. Then $\varphi:= s \circ f$ defines a non-zero Higgs field and by definition $\varphi^2= 0$, which proves that $E$ is wobbly.
 \end{proof}
 \begin{theorem}
The locus of wobbly bundles is of co-dimension one.
 \end{theorem}
 \begin{proof}
 Note that $X_1$ consists the elements of the form $(E, 0)$, where $0$ stands for the zero Higgs field. Thus the restriction of the forgetful map to $X_1$ is injective.
 On the other hand, the Hitchin map $h$ is $\mathbb{C}^*$-equivariant under the natural $\mathbb{C}^*$-action on $T^*\mathcal{M}_C^s(n, \delta)$. Thus, if $(p, \varphi)\in X_i$, then $(p, \lambda \varphi) \in X_i$ for any non-zero scalar $\lambda$. Since $X_i$ is closed, $(p, 0) \in X_i$. 
Thus, the image of $X_i$ under the forgetful map is the same as the image of $X_1 \cap X_i$. In other words, the image of $X_1 \cap X_i$ under the forgetful map is precisely $\mathcal{W}_i$.
 We need to show that $\mathcal{W}_i$ is of dimension $N-1$ for some $i$. \\
 Note that $h^{-1}(0)$ is a complete intersection, in particular locally complete intersection, Thus, by Hartshorne's connectedness theorem \cite[Theorem 3.4]{HC} one can rearrange the components if needed such that $X_i \cap X_{i+1}$ has codimension one in $X_i$. Thus, we have $X_1 \cap X_2$ is of co-dimension one in $X_1$. Since the restriction of the forgetful map to $X_1$ is injective, $\mathcal{W}_2$ is of co-dimension one.
 \end{proof}

\section{The Wobbly locus on the moduli  of vector bundles on curves}\label{S3}

Let $C$ be a smooth projective curve of genus $g \ge 2$ over the complex numbers.  Let $\mathcal{M}_C(n, \delta)$ be the moduli space of semi-stable vector bundles on $C$  of rank $r$  with fixed determinant $\delta$ of degree $d$. Let $h = \gcd (r, d)$.

\begin{definition}
(1) A {\it Higgs field} $\varphi$ on a vector bundle $E$ is a morphism $\varphi: E \to E \otimes K_C$. A Higgs field $\varphi$ is called {\it nilpotent} if the composition
$$
\varphi^m: E \xlongrightarrow{\varphi} E \otimes K_C \xlongrightarrow{\varphi \otimes \Id} E \otimes K_C^2 \xrightarrow{\makebox[1.0 cm]{$\varphi \otimes \Id^{2}$}} \cdots \xrightarrow{\makebox[1.5 cm]{$\varphi \otimes \Id^{m-1}$}} E \otimes K_C^m
$$
is zero for some $m \ge 1$.

(2) A vector bundle $E$  is called {\it very stable} if $E$ does not admit any nonzero nilpotent Higgs field. A vector bundle is called {\it wobbly}  if it is not very stable.
\end{definition}

Let $\varphi: E \to E \otimes K_C$ is a non-zero Higgs field and set $E^{i-1}:= \text{ker}(\varphi^i:E \to E \otimes K_C^i)$. Then we have a filtration
\begin{equation}\label{Mar31}
0=E^{-1} \subset E^0 \subset E^1 \subset E^2 \subset \cdots\subset E.
\end{equation}
Furthermore, if $E_{i+1}:= E^{i}/E^{i-1}$, then $\varphi$ induces non-zero injective homomorphisms $E_{i+1} \to E_i \otimes K_C$ for $i \ge 1$ \cite[P. 650 (1.4), (1.5)]{L}. Note that if $\varphi$ is nilpotent, then $\varphi^{m+1}=0$ for some $m \ge 1$. In other words, $E= E^m$ and there is an injection $E_{m+1} \to E_m \otimes K_C.$ Therefore, if $E$ is wobbly then, $E$ can be obtained as successive extensions of the form
 \begin{equation}\label{NL1}
\begin{tikzpicture}[baseline=(current  bounding  box.center)]
  \matrix(m)[matrix of math nodes,row sep=0.5em, column sep=1.5em,text height=1.5ex, 
 text depth=0.25ex]
 { 0 & E_1 = E^0 & E^1 & E_2 & 0 \\
 0 & E^1 & E^2 & E_3 & 0\\
 & \vdots & \vdots & \vdots & \\
 0 & E^{m-1} & E=E^m & E_{m+1} & 0\\};
 \path[->](m-1-1) edge (m-1-2);
 \path[->](m-1-2) edge (m-1-3);
 \path[->](m-1-3) edge (m-1-4);
 \path[->](m-1-4) edge (m-1-5);
 \path[->](m-2-1) edge (m-2-2);
 \path[->](m-2-2) edge (m-2-3);
 \path[->](m-2-3) edge (m-2-4);
 \path[->](m-2-4) edge (m-2-5);
 \path[->](m-4-1) edge (m-4-2);
 \path[->](m-4-2) edge (m-4-3);
 \path[->](m-4-3) edge (m-4-4);
 \path[->](m-4-4) edge (m-4-5);
 \end{tikzpicture}
\end{equation}
such that $H^0(C, E_{i+1}^* \otimes E_i  \otimes K_C) \ne 0$ for $i \ge 1$.


\begin{remark}\label{NR1}

Let $E$ be wobbly and let $m$ be the minimal integer such that $E$ admits a Higgs field $\varphi$ with $\varphi^{m+1}= 0$. Further, assume that $m \ge 2$.
\begin{enumerate}[(i)]
\item Suppose there is a non-zero homomorphism $E_{m+1} \to E^{m-i} \otimes K_C$ for some $i \ge 1$. Then it will induce a non-zero Higgs field $\psi: E \to E_{m+1} \to E^{m-i} \otimes K_C \to E \otimes K_C$ such that $\psi^k = 0$, for some $k \leq m$. Thus, by the minimality of $m$, no such non-zero homomorphism $E_{m+1} \to E^{m-i} \otimes K_C$ exits. Hence we have $H^0(E_{m+1}^* \otimes E^{m-i} \otimes K_C)= 0$ for $i \ge 1$.

\item Similarly, if there is a non-zero homomorphism $E_{j+1} \to E_{j-i} \otimes K_C, i \ge 1, j\ge 2$, then one can get a  Higgs field $\psi$ with $\psi^{m-i+1}= 0$, a contradiction. Thus, we may also assume that $H^0(E_{j+i}^* \otimes E_{j-i} \otimes K_C) =0$ for $i \ge 1$ and $j \ge 2$.

\item Let $E_{j+1} \to E_j$ be a non-zero homomorphism. Since $E_{j} \to E_{j-1} \otimes K_C$ injective, we get a non-zero homomorphism $E_{j+1} \to E_{j-1} \otimes K_C$, contradicting the above. Hence, no such non-zero homomorphism $E_{j+1} \longrightarrow E_j$ exists.

\item 
Note that $E_i \subset E^{i-1} \otimes K_C^{i-1}$ and  $\varphi$ induces a map $E \otimes K_C^{m-1} \to E_{m+1} \subset E \otimes K_C^m$. On the other hand, any morphism $\psi: E_{m+1} \to E_m \otimes K_C$ gives a map $\psi: E_{m+1} \to E^{m-1} \otimes K^m \subset E \otimes K_C^m$. Hence the composition $E \otimes K_C^{m-1} \to E_{m+1} \to E \otimes K_C^m$ gives a non-zero nilpotent Higgs field inducing the same filtration as of $\varphi$ on $E$. Thus the non-zero Higgs fields of $E$ inducing the filtration (\ref{Mar31}) on $E$, is same as $H^0(E_{m+1}^* \otimes E_m \otimes K_C)$.

\end{enumerate}

\end{remark}

\subsection{Rational curves through E} Let $E \in \mathcal{M}_C(n, \delta)$ be a stable wobbly bundle. Then $E$ admits a non-zero nilpotent Higgs field 
 \begin{equation}\label{Mar27}
 \varphi: E \to E \otimes K_C
 \end{equation}
and $\varphi$ induces a filtration on $E$ as in (\ref{Mar31}). Let $m$ be the smallest integer such that $\varphi^{m+1}= 0$. Then, $E$ can be obtained by successive extensions, as in (\eqref{NL1}). Here we will briefly recall the construction done in \cite[Lemma 3.6]{TB} of a rational curve in $\mathcal{M}_C(n, \delta)$ through $E$.

Let $\mathcal{P}_1= \mathbb{P}({\text{Ext}}^1(E_2, E_1))$. Then there is a canonical extension on $C \times \mathbb{P}^1 \times \mathcal{P}_1$:
 \begin{equation}\label{NE1}
 0 \longrightarrow p_1^*E_1 \otimes p_2^*{\mathcal{O}(1)} \longrightarrow \mathcal{E}_1^{\prime} \longrightarrow p_1^*E_2 \longrightarrow 0,
 \end{equation}
where $p_i$ is the $i$-th coordinate projection. Now consider the projective bundle  $\mathcal{P}_2= \mathbb{P}(R^1{p_3}_*\mathcal{H}\text{om}(p_1^*E_3, \mathcal{E}_1^{\prime} \otimes p_2^*\mathcal{O}(1))$ over $\mathcal{P}_1$.  Inductively, we construct a projective bundle $\mathcal{P}_j = \mathbb{P}(R^1{p_3}_*\mathcal{H}\text{om}(p_1^*E_{j+1}, \mathcal{E}_{j-1}^{\prime} \otimes p_2^*\mathcal{O}(1))$ over $\mathcal{P}_{j-1}, 2 \leq j \leq m$ together with a canonical extension on $C \times \mathbb{P}^1 \times \mathcal{P}_j$ as follows:
Let $\pi_{j-1}: \mathcal{P}_j \to \mathcal{P}_{j-1}$ be the natural projection map. We denote the morphism,  $\text{Id}_{C \times \mathbb{P}^1} \times \pi_{j-1}:  C \times \mathbb{P}^1 \times \mathcal{P}_j \to C \times \mathbb{P}^1 \times \mathcal{P}_{j-1}$, by $\Phi_{j-1}$. Then we have the canonical extension on $C \times \mathbb{P}^1 \times \mathcal{P}_j$
 \begin{equation}\label{NE2}
0 \longrightarrow \Phi_{j-1}^*(\mathcal{E}_{j-1}^{\prime}) \otimes p_2^*\mathcal{O}(1) \longrightarrow \mathcal{E}_j^{\prime} \longrightarrow p_1^*E_{j+1} \longrightarrow 0.
\end{equation}
The restriction of $\mathcal{E}_{m}^{\prime}$ to the fiber over a point $p \in \mathcal{P}_m$ gives a vector bundle on $C \times \mathbb{P}^1 \times \{p\}$.  Let $ x_m \in \mathcal{P}_m$ be such that the corresponding rational curve passes through $E$.   Since, $E$ is stable, a general point in the rational curve is stable, which gives a rational map: 
\begin{equation}
f: \mathbb{P}^1 \cdots \to \mathcal{M}_C(n, \delta).
\end{equation}
Since the moduli space of equivalence classes of semistable vector bundles is complete, the above map extends to a morphism, which gives a rational curve
\begin{equation}\label{NE3}
f: \mathbb{P}^1 \longrightarrow \mathcal{M}_C(n, \delta).
\end{equation}
Let $x_j \in \mathcal{P}_j$ be such that the rational curve corresponding to $x_j$ passes through $E^j$.We denote the restriction of $\mathcal{E}_j^{\prime}$ to $C \times \mathbb{P}^1 \times \{x_j\}$ by $\mathcal{E}_j$. So we got two exact sequences of vector bundles on $C \times \mathbb{P}^1$:
\begin{equation}\label{NE11}
 0 \longrightarrow p_1^*E_1 \otimes p_2^*{\mathcal{O}(1)} \longrightarrow \mathcal{E}_1 \longrightarrow p_1^*E_2 \longrightarrow 0,
\end{equation}
and
 \begin{equation}\label{NE21}
0 \longrightarrow \mathcal{E}_{j-1} \otimes p_2^*\mathcal{O}(1) \longrightarrow \mathcal{E}_j \longrightarrow p_1^*E_{j+1} \longrightarrow 0.
\end{equation}

Let $r_i, d_i$ be the rank and degree of $E_i$. Then the anticanonical degree of the rational curve is $ q:=2\sum_{i < j}(r_id_j-r_jd_i)(j-i)$ \cite[Lemma 3.6]{TB}.

We compute several cohomological relations of bundles occurring in \eqref{NE11} and \eqref{NE21} in the following Lemma, which we need in the next Proposition. Throughout the proof, we will use K\"{u}nneth formula to compute the cohomology on the product: If $X$ and $Y$ are two projective varieties and $E, F$ are quasi-coherent sheaves on $X$ and $Y$, respectively, then $H^i(X \times Y, p_1^*E \otimes p_2^*F)= \oplus_{p+q=i}H^p(X, E) \otimes H^q(Y, F)$ to compute the cohomogies $H^i(X \times Y, p_1^*E \otimes p_2^*F)$ where $p_1, p_2$ are the projection of $X \times Y$ to the first and second factors respectively.

\begin{lemma}\label{nnl1}
Consider the  exact sequences of vector bundles on $C \times \mathbb{P}^1$, in \eqref{NE11} and \eqref{NE21}. Let $F$ be a vector bundles on $C$ and $k \geq 0$ be an integer. Then
\begin{enumerate}[$($a$)$]
\item $H^0(\mathcal{E}_j^* \otimes p_1^*F \otimes p_2^*\mathcal{O}(-1))= H^0(\mathcal{E}_j \otimes p_1^*F \otimes p_2^*\mathcal{O}(-j-1))= 0$ and $H^2(\mathcal{E}_j \otimes p_1^*F \otimes p_2^*\mathcal{O}(k))= 0$.

\vspace{0.1 cm}

\item we have 
\begin{align*}
H^2(\mathcal{E}_j^* \otimes p_1^*F \otimes p_2^*\mathcal{O}(-1)) &= H^2(\mathcal{E}_{j-1}^* \otimes p_1^*F \otimes p_2^*\mathcal{O}(-2))\\
H^1(\mathcal{E}_j^* \otimes p_1^*F \otimes p_2^*\mathcal{O}(-1)) &= H^1(\mathcal{E}_{j-1}^* \otimes p_1^*F \otimes p_2^*\mathcal{O}(-2))
\end{align*}
Further more, if $j \ge 2$, then
$$
H^2(\mathcal{E}_j^* \otimes p_1^*E_{j+1} \otimes p_2^*\mathcal{O}(-1))=H^2(\mathcal{E}_{j-1}^* \otimes p_1^*E_{j+1} \otimes p_2^*\mathcal{O}(-2))=0.
$$

\vspace{0.1 cm}

\item $H^2(\mathcal{E}_{j-1} \otimes \mathcal{E}_j^*)=  H^2(\text{Ad}(\mathcal{E}_{j-1}) \otimes p_2^*\mathcal{O}(-1))$.

\vspace{0.1 cm}

\item if $j \ge 2,$ we have $h^2(\text{Ad}(\mathcal{E}_j) \otimes p_2^*\mathcal{O}(-1)) =  
   h^0(C, E_j \otimes E_{j+1}^* \otimes K_C).$

\vspace{0.1 cm}

\item for $j \ge 2$, we have 
\begin{align*}
H^1(\mathcal{E}_{j-1} \otimes p_1^*E_{j+1}^* \otimes p_2^*\mathcal{O}(k)) = (H^1(&C, E_j \otimes E_{j+1}^*) \otimes H^0(\mathbb{P}^1, \mathcal{O}(k)))\\
&\oplus H^1(\mathcal{E}_{j-2} \otimes p_1^*E_{j+1}^* \otimes p_2^*\mathcal{O}(k+1)).
\end{align*}
In particular,
$$
h^1(\mathcal{E}_{j-1} \otimes p_1^*E_{j+1}^*) =  \sum_{k=0}^{j-1} (h^1(C, E_{j-k} \otimes E_{j+1}^*) \cdot h^0(\mathbb{P}^1, \mathcal{O}(k))).
$$

\vspace{0.1 cm}

\item if $j \ge 2$, we have 
\begin{align*}
h^1(\mathcal{E}_{j-1}^* \otimes p_1^*E_{j+1} \otimes p_2^*\mathcal{O}(-2)) &= - h^1(C, E_j^* \otimes E_{j+1})\\
\sum_{k=0}^{j-1}h^0(C, E_{j-k}^* &\otimes E_{j+1}) \cdot h^1(\mathbb{P}^1, \mathcal{O}(-2-k)).
\end{align*}

\vspace{0.1 cm}

\item $H^0(\text{Ad}(\mathcal{E}_j) \otimes p_2^*\mathcal{O}(-1))= 0$.
\end{enumerate}
\end{lemma}
\begin{proof}
$(a)$ We will use induction on $j \geq 1$. Let $j=1$. Taking the dual exact sequence of \eqref{NE11}  and tensoring by $p_1^*F \otimes p_2^*\mathcal{O}(-1)$, we have 
\[
0 \to p_1^*E_2^* \otimes  p_1^*F \otimes p_2^*\mathcal{O}(-1) \to \mathcal{E}_1^* \otimes p_1^*F \otimes p_2^*\mathcal{O}(-1) \to p_1^*E_1^* \otimes
p_1^*F \otimes p_2^*\mathcal{O}(-2) \to 0.
\]
Considering the cohomology sequence we get
\begin{align*}
0 \to H^0(p_1^*E_2^* \otimes  p_1^*F \otimes p_2^*\mathcal{O}(-1)) &\to H^0(\mathcal{E}_1^* \otimes p_1^*F \otimes p_2^*\mathcal{O}(-1))\\
&\to H^0(p_1^*E_1^* \otimes p_1^*F \otimes p_2^*\mathcal{O}(-2)).
\end{align*}
We have, $H^0(p_1^*E_2^* \otimes  p_1^*F \otimes p_2^*\mathcal{O}(-1))=H^0(p_1^*E_1^* \otimes
p_1^*F \otimes p_2^*\mathcal{O}(-2))=0$ by K\"{u}nneth formula. Thus
\begin{equation}\label{ne10}
H^0(\mathcal{E}_1^* \otimes p_1^*F \otimes  p_2^*\mathcal{O}(-1))= 0.
\end{equation}
Similarly, tensoring \eqref{NE11} by $p_1^*F \otimes p_2^*\mathcal{O}(k)$ we have,
\[
0 \to p_1^*E_1 \otimes  p_1^*F \otimes p_2^*\mathcal{O}(k+1) \to \mathcal{E}_1 \otimes p_1^*F \otimes p_2^*\mathcal{O}(k)  \to p_1^*E_2 \otimes p_1^*F \otimes p_2^*\mathcal{O}(k)\to 0.
\]
Note that, for any integer $k \geq 0$
\begin{align*}
H^2(p_1^*E_1 \otimes p_1^*F \otimes p_2^*\mathcal{O}(k+1)) = 0 = H^2(p_1^*E_2 \otimes
p_1^*F \otimes p_2^*\mathcal{O}(k))
\end{align*}
and for any integer $k \leq -2$.
\begin{align*}
H^0(p_1^*E_1 \otimes p_1^*F \otimes p_2^*\mathcal{O}(k+1)) = 0 = H^0(p_1^*E_2 \otimes p_1^*F \otimes p_2^*\mathcal{O}(k))
\end{align*}
Thus $H^0(\mathcal{E}_1 \otimes p_1^*F \otimes p_2^*\mathcal{O}(k)) = 0$, for $k \leq -2$ and $H^2(\mathcal{E}_1 \otimes p_1^*F \otimes p_2^*\mathcal{O}(k)) = 0$, for $k \geq 0$. This shows that $(a)$ holds for $j = 1$.

Tensoring  by $p_1^*F \otimes p_2^*\mathcal{O}(-1)$, the dual exact sequence of \eqref{NE21}, we have,
\begin{align*}
0 \to p_1^*E_{j+1}^* \otimes p_1^*F \otimes p_2^*\mathcal{O}(-1) &\to \mathcal{E}_j^* \otimes p_1^*F \otimes p_2^*\mathcal{O}(-1)\\
&\to \mathcal{E}_{j-1}^* \otimes  p_1^*F \otimes p_2^*\mathcal{O}(-2) \to 0.
\end{align*}
By induction hypothesis, $H^0(\mathcal{E}_{j-1}^* \otimes p_1^*F  \otimes p_2^*\mathcal{O}(-1))= 0$ so that we get $H^0(\mathcal{E}_{j-1}^* \otimes p_1^*F \otimes p_2^*\mathcal{O}(-2))= 0$. Since $H^0(p_1^*E_{j+1}^* \otimes p_1^*F \otimes p_2^*\mathcal{O}(-1)) = 0$ by K\"{u}nneth formula, we get $H^0(\mathcal{E}_j^* \otimes p_2^*\mathcal{O}(-1) \otimes p_1^*F)=0$. Tensoring the exact sequence \eqref{NE21} by $p_1^*F \otimes p_2^*\mathcal{O}(k)$, where $k \in \mathbb{Z}$, we have
\[
0 \to \mathcal{E}_{j-1} \otimes  p_1^*F \otimes  p_2^*\mathcal{O}(k+1) \to \mathcal{E}_j \otimes p_1^*F \otimes p_2^*\mathcal{O}(k) \to p_1^*E_{j+1} \otimes p_1^*F \otimes p_2^*\mathcal{O}(k) \to 0.
\]
Since for any $k \geq 0$,
$$
H^2(p_1^*E_{j+1} \otimes p_1^*F \otimes p_2^*\mathcal{O}(k))= 0 = H^2(\mathcal{E}_{j-1} \otimes p_1^*F \otimes p_2^*\mathcal{O}(k+1)),
$$
by induction hypothesis, we have $H^2(\mathcal{E}_j \otimes p_1^*F \otimes p_2^*\mathcal{O}(k))= 0$, for any $k \geq 0$. On the other hand, $H^0(p_1^*E_{j+1} \otimes p_1^*F \otimes p_2^*\mathcal{O}(k))= 0$ if $k < 0$ and $H^0(\mathcal{E}_{j-1} \otimes p_1^*F \otimes p_2^*\mathcal{O}(-j))= 0$, by induction hypothesis. Combining, we get $H^0(\mathcal{E}_j \otimes p_1^*F \otimes p_2^*\mathcal{O}(-j-1))= 0$.

\vspace{0.2 cm}

$(b)$ Tensoring the dual exact sequence of \eqref{NE21} by $p_1^*F \otimes p_2^*\mathcal{O}(-1)$, we have
\begin{align*}
0 \to p_1^*(E_{j+1}^*\otimes F) \otimes p_2^*\mathcal{O}(-1) &\to \mathcal{E}_j^* \otimes p_1^*F \otimes p_2^*\mathcal{O}(-1)\\
&\to \mathcal{E}_{j-1}^* \otimes p_1^*F \otimes p_2^*\mathcal{O}(-2) \to 0.
\end{align*}
Since $H^2(p_1^*(E_{j+1}^*\otimes F) \otimes p_2^*\mathcal{O}(-1)) = 0 = H^1(p_1^*(E_{j+1}^*\otimes F) \otimes p_2^*\mathcal{O}(-1))$ by K\"{u}nneth formula, considering the cohomology sequence we have the first part of $(b)$.

By Serre duality, we have 
$$
H^2(\mathcal{E}_j^* \otimes p_1^*E_{j+1} \otimes p_2^*\mathcal{O}(-1))= H^0(\mathcal{E}_j \otimes p_1^*(E_{j+1}^* \otimes K_C) \otimes p_2^*\mathcal{O}(-1))^{\vee}.
$$ 
Tensoring \eqref{NE21} by $p_1^*K_C \otimes p_2^*\mathcal{O}(-1)$, we get
\[
0 \to \mathcal{E}_{j-1} \otimes p_1^*K_C  \to \mathcal{E}_j \otimes p_1^*K_C \otimes p_2^*\mathcal{O}(-1) \to p_1^*E_{j+1}\otimes p_1^*K_C \otimes p_2^*\mathcal{O}(-1) \to 0. 
\]
Since $H^0(\text{End}(p_1^*E_{j+1}) \otimes p_1^*K_C \otimes p_2^*\mathcal{O}(-1))= 0$ by K\"{u}nneth formula, any non-zero homomorphism $p_1^*E_{j+1} \to \mathcal{E}_j \otimes p_1^*K_C \otimes p_2^*\mathcal{O}(-1) $ will factor through a non-zero homomorphism $p_1^*E_{j+1} \to \mathcal{E}_{j-1} \otimes p_1^*K_C$. If $j \ge 2$, then  as in remark \eqref{NR1}, a nonzero homomorphism  $p_1^*E_{j+1} \to \mathcal{E}_{j-1} \otimes p_1^*K_C$ will give a nilpotent  homomorphism $\tilde{\varphi}: \mathcal{E}_j \to \mathcal{E}_j \otimes p_1^*{K_C} \otimes p_2^*\mathcal{O}(-1)$ with ${\tilde{\varphi}}^2=0$, contradicting the minimality of $m$. Thus $H^2(\mathcal{E}_j^* \otimes p_1^*E_{j+1} \otimes p_2^*\mathcal{O}(-1))= 0$. By taking $F= E_{j+1}$
in the first part of (b) we get the last part.

\vspace{0.2 cm}

$(c)$ Tensoring the dual sequence of \eqref{NE21} by $\mathcal{E}_{j-1}$, we have
\[
0 \to \mathcal{E}_{j-1} \otimes p_1^*E_{j+1}^*  \to \mathcal{E}_{j-1} \otimes \mathcal{E}_j^* \to (\text{Ad}(\mathcal{E}_{j-1}) \oplus \mathcal{O}) \otimes p_2^*\mathcal{O}(-1)\to 0. 
\]
By $(a)$,  $H^2(\mathcal{E}_{j-1} \otimes p_1^*E_{i+1}^*)= 0$ and hence
$$
H^2(\mathcal{E}_{j-1} \otimes \mathcal{E}_j^*)= H^2(\text{Ad}(\mathcal{E}_{j-1}) \otimes p_2^*\mathcal{O}(-1)).
$$

\vspace{0.2 cm}

$(d)$ Since $H^2(\text{Ad}(\mathcal{E}_j) \otimes p_2^*\mathcal{O}(-1))^{\vee} = H^0(\text{Ad}(\mathcal{E}_j) \otimes p_1^*K_C \otimes p_2^*\mathcal{O}(-1))$, it is enough to show that
\[
h^0(\text{Ad}(\mathcal{E}_j) \otimes p_1^*K_C \otimes p_2^*\mathcal{O}(-1)) = h^0(C, E_{j} \otimes E_{j+1}^* \otimes K_C).
\]
First we claim that if $\psi \in H^0(\text{Ad}(\mathcal{E}_j) \otimes  p_1^*K_C \otimes p_2^*\mathcal{O}(-1))$, then $\psi$ is nilpotent. In other words,
\begin{align*}
\mathcal{E}_j \xrightarrow{\psi} \mathcal{E}_j \otimes  p_1^*K_C \otimes p_2^*\mathcal{O}(-1) &\xrightarrow{\psi \otimes \Id} \mathcal{E}_j \otimes  p_1^*K_C^2 \otimes p_2^*\mathcal{O}(-2) \longrightarrow\\
&\cdots \xrightarrow{\makebox[1.5 cm]{$\psi \otimes \Id^{\ell -1}$}} \mathcal{E}_j \otimes  p_1^*K_C^{\ell} \otimes p_2^*\mathcal{O}(-\ell)
\end{align*}
is zero for some $l \ge 1$. We will show that  $H^0((\text{Ad}(\mathcal{E}_j) \otimes p_1^*K_C^{\ell} \otimes p_2^*{\mathcal O}(-\ell)) = 0$ for some $\ell \in \mathbb{Z}$.

We will use induction on $j$. Let $j=1$. Tensoring (\ref{NE11}) by $\mathcal{E}_1^* \otimes p_1^*K_C^2 \otimes p_2^*\mathcal{O}(-2)$, we get
\begin{align*}
0 \to \mathcal{E}_1^* \otimes p_1^*(E_1 \otimes K_C^2) \otimes p_2^*\mathcal{O}(-1)  &\to \text{End}(\mathcal{E}_1) \otimes p_1^*{K_C^2} \otimes p_2^*\mathcal{O}(-2)\\
&\to \mathcal{E}_1^* \otimes p_1^*(E_2 \otimes K_C^2) \otimes p_2^*\mathcal{O}(-2)  \to 0.
\end{align*}
By $(a)$ we have,
$$
H^0(p_1^*(\mathcal{E}_1^* \otimes E_1 \otimes K_C^2) \otimes p_2^*\mathcal{O}(-1)) = 0 = H^0(\mathcal{E}_1^* \otimes p_1^*(E_2 \otimes K_C^2) \otimes p_2^*\mathcal{O}(-2)).
$$
Thus $H^0(\text{End}(\mathcal{E}_1) \otimes p_1^*{K_C^2} \otimes p_2^*\mathcal{O}(-2)) = 0$ and hence $H^0(\text{Ad}(\mathcal{E}_1) \otimes p_1^*{K_C^2} \otimes p_2^*\mathcal{O}(-2)) = 0$.

Let us now assume that $H^0(\text{Ad}(\mathcal{E}_{j-1}) \otimes p_1^*K_C^{\ell} \otimes p_2^*{\mathcal O}(-\ell)) = 0$ for some integer $\ell$. Tensoring (\ref{NE21}) by $\mathcal{E}_j^* \otimes p_1^*K^{m} \otimes p_2^*\mathcal{O}(-m)$, where $m > \text{max}\{l, j\}$, we get
\begin{align*}
0 \to \mathcal{E}_{j-1} \otimes \mathcal{E}_j^* \otimes p_1^*K_C^{m} &\otimes p_2^*\mathcal{O}(-m+1)  \to (\text{Ad}(\mathcal{E}_j) \oplus \mathcal{O}) \otimes p_1^*{K_C^{m}} \otimes p_2^*\mathcal{O}(-m)\\
&\to \mathcal{E}_j^* \otimes p_1^*(E_{j+1} \otimes K_C^{m}) \otimes p_2^*\mathcal{O}(-m) \to 0.
\end{align*}
By $(a)$, $H^0(\mathcal{E}_j^* \otimes p_1^*(E_{j+1} \otimes K_C^{m}) \otimes p_2^*\mathcal{O}(-m))= 0$. On the other hand, tensoring the dual exact sequence of (\ref{NE21}), by $\mathcal{E}_{j-1} \otimes p_1^*K_C^{m} \otimes p_2^*\mathcal{O}(-m+1)$, we get
\begin{align*}
0 \to \mathcal{E}_{j-1} \otimes p_1^*(E_{j+1}^* \otimes K_C^{m}) \otimes p_2^*&\mathcal{O}(-m+1)  \to \mathcal{E}_{j-1} \otimes \mathcal{E}_j^* \otimes p_1^*K_C^{m} \otimes p_2^*\mathcal{O}(-m+1)\\
&\to (\text{Ad}(\mathcal{E}_{j-1}) \oplus \mathcal{O}) \otimes p_1^*{K_C^{m}} \otimes p_2^*\mathcal{O}(-m)  \to 0.
\end{align*}
Since $m > j,$ by (a) and by induction hypothesis, we have
\[
H^0(\mathcal{E}_{j-1} \otimes p_1^*(E_{j+1}^* \otimes K_C^{m}) \otimes p_2^*\mathcal{O}(-m+1)) = 0 = H^0 \big( (\text{Ad}(\mathcal{E}_{j-1}) \oplus \mathcal{O}) \otimes p_1^*{K_C^{m}} \otimes p_2^*\mathcal{O}(-m) \big).
\]
Thus we have, 
\[
H^0(\text{Ad}(\mathcal{E}_{j})  \otimes p_1^*{K_C^{m}} \otimes p_2^*\mathcal{O}(-m))= 0. 
\]

Let $\psi \in H^0(\text{Ad}(\mathcal{E}_{j})  \otimes p_1^*K_C \otimes p_2^*\mathcal{O}(-1))$. Consider the diagram
 \begin{equation}\label{RE90}
\begin{tikzpicture}[baseline=(current  bounding  box.center)][font=\small]
  \matrix(m)[matrix of math nodes,row sep=2.5em, column sep=1.0em,text height=1.5ex, 
 text depth=0.5ex]
 { 0 & \mathcal{E}_{j-1} & \mathcal{E}_j \otimes p_2^*\mathcal{O}(-1) & p_1^*E_{j+1} \otimes p_2^*\mathcal{O}(-1) & 0 \\
 0 & \mathcal{E}_{j-1} \otimes p_1^*K_C \otimes p_2^*\mathcal{O}(-1) & \mathcal{E}_j \otimes p_1^*K_C \otimes p_2^*\mathcal{O}(-2) & p_1^*E_{j+1} \otimes p_1^*K_C\otimes p_2^*\mathcal{O}(-2) & 0.\\};
 \path[->](m-1-1) edge (m-1-2);
 \path[->](m-1-2) edge (m-1-3);
 \path[->](m-1-3) edge (m-1-4);
 \path[->](m-1-4) edge (m-1-5);
 \path[->](m-2-1) edge (m-2-2);
 \path[->](m-2-2) edge (m-2-3);
 \path[->](m-2-3) edge (m-2-4);
 \path[->](m-2-4) edge (m-2-5);
 \path[->](m-1-3) edge node[right]{${\psi}$} (m-2-3);
 \end{tikzpicture}
\end{equation}
By $(a)$, $H^0(\mathcal{E}_{j-1} \otimes p_1^*(E_{j+1} \otimes K_C) \otimes p_2^*\mathcal{O}(-2))= 0$ so the restriction of ${\psi}$ to $\mathcal{E}_{j-1}$ gives a nilpotent endomorphism $\psi_{j-1}$  on $\mathcal{E}_{j-1}$. Taking $j=1$ in \eqref{RE90}, we have
  \begin{equation}\label{RE89}
\begin{tikzpicture}[baseline=(current  bounding  box.center)][font=\small]
  \matrix(m)[matrix of math nodes,row sep=2.5em, column sep=1.0em,text height=1.5ex, 
 text depth=0.5ex]
 { 0 & p_1^*E_1 &  \mathcal{E}_1 \otimes p_2^*\mathcal{O}(-1) & p_1^*E_2 \otimes p_2^*\mathcal{O}(-1) & 0 \\
 0 & p_1^*E_1 \otimes p_1^*K_C \otimes p_2^*\mathcal{O}(-1) & \mathcal{E}_1 \otimes p_1^*K_C \otimes p_2^*\mathcal{O}(-2) & p_1^*E_2 \otimes p_1^*K_C \otimes p_2^*\mathcal{O}(-2) & 0.\\};
 \path[->](m-1-1) edge (m-1-2);
 \path[->](m-1-2) edge (m-1-3);
 \path[->](m-1-3) edge (m-1-4);
 \path[->](m-1-4) edge (m-1-5);
 \path[->](m-2-1) edge (m-2-2);
 \path[->](m-2-2) edge (m-2-3);
 \path[->](m-2-3) edge (m-2-4);
 \path[->](m-2-4) edge (m-2-5);
 \path[->](m-1-3) edge node[right]{$\psi_1$} (m-2-3);
 \end{tikzpicture}
\end{equation}
Again by $(a)$, $H^0(\mathcal{E}_1 \otimes p_1^*(E_1^* \otimes K_C) \otimes p_2^*\mathcal{O}(-2))= 0$ so that $p_1^*E_1$ is in the kernel of $\psi_1$. Since $H^0(p_1^*(\text{End}(E_2) \otimes K_C) \otimes p_2^*\mathcal{O}(-1))=0$, we get $\psi_1^2= 0$. Taking $j=2$ in \eqref{RE90}, we get
\begin{equation}\label{RE88}
\begin{tikzpicture}[baseline=(current  bounding  box.center)][font=\small]
  \matrix(m)[matrix of math nodes,row sep=2.5em, column sep=1.0em,text height=1.5ex, 
 text depth=0.5ex]
 { 0 & \mathcal{E}_1 &  \mathcal{E}_2 \otimes p_2^*\mathcal{O}(-1) & p_1^*E_3 \otimes p_2^*\mathcal{O}(-1) & 0 \\
 0 & \mathcal{E}_1 \otimes p_1^*K_C^2 \otimes p_2^*\mathcal{O}(-2) & \mathcal{E}_2 \otimes p_1^*K_C^2 \otimes p_2^*\mathcal{O}(-3) & p_1^*E_3 \otimes p_1^*K_C^2\otimes p_2^*\mathcal{O}(-3) & 0.\\};
 \path[->](m-1-1) edge (m-1-2);
 \path[->](m-1-2) edge (m-1-3);
 \path[->](m-1-3) edge (m-1-4);
 \path[->](m-1-4) edge (m-1-5);
 \path[->](m-2-1) edge (m-2-2);
 \path[->](m-2-2) edge (m-2-3);
 \path[->](m-2-3) edge (m-2-4);
 \path[->](m-2-4) edge (m-2-5);
 \path[->](m-1-3) edge node[right]{$\psi_2^2$} (m-2-3);
 \end{tikzpicture}
\end{equation}
By previous step $\mathcal{E}_1$ is in the kernel of $\psi_2^2$ and since $H^0(p_1^*(\text{End}(E_3) \otimes K_C^2) \otimes p_2^*\mathcal{O}(-2))=0$, $\psi_2^3= 0$. Inductively we have,  $\psi_i^{i+1}=0$. Thus $\psi$ induces the same filtration on $\mathcal{E}_j,$ as the filtration obtained by the Higgs field $\varphi_{\mid_{E^j}}$ on $E^j$ in (\ref{Mar31}). Hence, there is a bijection between $H^0(\text{Ad}(\mathcal{E}_j) \otimes p_1^*K_C \otimes p_2^*\mathcal{O}(-1))$ and the nilpotent Higgs fields of $E^j$, which gives the filtration of $E^j$ as in (\ref{Mar31}), which by remark \ref{NR1} is same as $H^0(C, E_{j} \otimes E_{j+1}^* \otimes K_C)$ .

\vspace{0.2 cm}

$(e)$ Tensoring the short exact sequence 
\[
0 \to \mathcal{E}_{j-2} \otimes p_2^*\mathcal{O}(1) \to \mathcal{E}_{j-1}  \to p_1^*E_j \to 0
\]
(see (\ref{NE21})) by $p_1^*E_{j+1}^* \otimes p_2^*\mathcal{O}(k)$, $k \geq 0$ any integer, we get
\begin{align*}
0 \to \mathcal{E}_{j-2} \otimes p_1^*E_{j+1}^* \otimes p_2^*\mathcal{O}(k+1) &\to \mathcal{E}_{j-1}  \otimes p_1^*E_{j+1}^* \otimes p_2^*\mathcal{O}(k)\\
&\to p_1^*(E_j \otimes E_{j+1}^*) \otimes p_2^*\mathcal{O}(k)\to 0.
\end{align*}
We have $H^0(C, E_j \otimes E_{j+1}^*)= 0 = H^2(\mathcal{E}_{j-2} \otimes p_2^*\mathcal{O}(k+1) \otimes p_1^*E_{j+1}^*)$ by By Remark \eqref{NR1} and $(a)$. Thus
\begin{align*}
H^1(\mathcal{E}_{j-1}  \otimes p_1^*E_{j+1}^* \otimes p_2^*\mathcal{O}(k)) &= H^1(\mathcal{E}_{j-2} \otimes p_1^*E_{j+1}^* \otimes p_2^*\mathcal{O}(k+1))\\
&\oplus \big( H^1(C, E_j \otimes E_{j+1}^*) \otimes H^0(\mathbb{P}^1, \mathcal{O}(k)) \big).
\end{align*}

\vspace{0.2 cm}

$(f)$ Tensoring the dual of the exact sequence 
\[
0 \to \mathcal{E}_{j-2} \otimes p_2^*\mathcal{O}(1) \to \mathcal{E}_{j-1}  \to p_1^*E_j \to 0
\]
(see (\ref{NE21})) by $p_1^*E_{j+1} \otimes p_2^*\mathcal{O}(-2)$, we get
\begin{align*}
0 \to p_1^*(E_{j}^* \otimes E_{j+1}) \otimes p_2^*\mathcal{O}(-2) &\to \mathcal{E}_{j-1}^* \otimes p_1^*E_{j+1} \otimes p_2^*\mathcal{O}(-2)\\
&\to \mathcal{E}_{j-2}^* \otimes p_1^*E_{j+1} \otimes p_2^*\mathcal{O}(-3) \to 0.
\end{align*}
By $(b)$, $H^2(\mathcal{E}_{j-1}^* \otimes p_1^*E_{j+1} \otimes p_2^*\mathcal{O}(-2)) = 0$. By $(a)$, $H^0(\mathcal{E}_{j-2}^* \otimes p_1^*E_{j+1} \otimes p_2^*\mathcal{O}(-1)) = 0$ and hence we have $H^0(\mathcal{E}_{j-2}^* \otimes p_1^*E_{j+1} \otimes p_2^*\mathcal{O}(-3)) = 0$. Thus
\begin{align*}
h^1(\mathcal{E}_{j-1}^* &\otimes p_1^*E_{j+1} \otimes p_2^*\mathcal{O}(-2)) = h^1(p_1^*(E_{j}^* \otimes E_{j+1}) \otimes p_2^*\mathcal{O}(-2))\\
&+ h^1(\mathcal{E}_{j-2}^* \otimes p_1^*E_{j+1} \otimes p_2^*\mathcal{O}(-3)) - h^2(p_1^*(E_{j}^* \otimes E_{j+1}) \otimes p_2^*\mathcal{O}(-2))\\
&= h^0(C, E_{j}^* \otimes E_{j+1}) \cdot h^1(\mathbb{P}^1, \mathcal{O}(-2) + h^1(\mathcal{E}_{j-2}^* \otimes p_1^*E_{j+1} \otimes p_2^*\mathcal{O}(-3))\\
& \,\,\,\,\,\,\,\,\,\,\, - h^1(C, E_j^* \otimes E_{j+1}).
\end{align*}
Similarly, tensoring the dual of the exact sequence 
\[
0 \to \mathcal{E}_{j-3} \otimes p_2^*\mathcal{O}(1) \to \mathcal{E}_{j-2}  \to p_1^*E_{j-1} \to 0
\]
(see (\ref{NE21})) by $p_1^*E_{j+1} \otimes p_2^*\mathcal{O}(-3)$, we get
\begin{align*}
0 \to p_1^*(E_{j-1}^* \otimes E_{j+1}) \otimes p_2^*\mathcal{O}(-3) &\to \mathcal{E}_{j-2}^* \otimes p_1^*E_{j+1} \otimes p_2^*\mathcal{O}(-3)\\
&\to \mathcal{E}_{j-3}^* \otimes p_1^*E_{j+1} \otimes p_2^*\mathcal{O}(-4) \to 0.
\end{align*}
By $(a)$ and $(b)$, 
\[
H^0(\mathcal{E}_{j-3}^* \otimes p_1^*E_{j+1} \otimes p_2^*\mathcal{O}(-4)) = 0 = H^2(\mathcal{E}_{j-2}^* \otimes p_1^*E_{j+1} \otimes p_2^*\mathcal{O}(-3)).
\] 
Since
\begin{align*}
h^2(p_1^*(E_{j-1}^* \otimes E_{j+1}) \otimes p_2^*\mathcal{O}(-3)) &= h^1(C, E_{j-1}^* \otimes E_{j+1}) \cdot h^1(\mathbb{P}^1, \mathcal{O}(-3))\\
&= h^0(C, E_{j-1} \otimes E_j^* \otimes K_C)  \cdot h^1(\mathbb{P}^1, \mathcal{O}(-3)) = 0
\end{align*}
by Remark \ref{NR1}, we have
\begin{align*}
h^1(\mathcal{E}_{j-2}^* \otimes p_1^*E_{j+1} \otimes p_2^*\mathcal{O}(-3)) = h^0(C, E_{j-1}^* &\otimes E_{j+1}) \cdot h^1(\mathbb{P}^1, \mathcal{O}(-3))\\
&+ h^1(\mathcal{E}_{j-3}^* \otimes p_1^*E_{j+1} \otimes p_2^*\mathcal{O}(-4))
\end{align*}
Inductively, we can write
\[
h^1(\mathcal{E}_{j-2}^* \otimes p_1^*E_{j+1} \otimes p_2^*\mathcal{O}(-3))= \sum_{i=0}^{j-2}h^0(C, E_{j-1-i}^* \otimes E_{j+1}) \cdot h^1(\mathbb{P}^1, \mathcal{O}(-3-i))
\]
so that
\begin{align*}
h^1(\mathcal{E}_{j-1}^* \otimes p_1^*E_{j+1} \otimes p_2^*\mathcal{O}(-2)) &= \sum_{i=0}^{j-2}h^0(C, E_{j-1-i}^* \otimes E_{j+1}) \cdot h^1(\mathbb{P}^1, \mathcal{O}(-3-i))\\
&+ h^0(C, E_{j}^* \otimes E_{j+1}^*)- h^1(C, E_{j}^* \otimes E_{j+1}).
\end{align*}

\vspace{0.2 cm}

$(g)$  Tensoring  $0 \to \mathcal{E}_{j-1} \otimes p_2^*\mathcal{O}(1) \to \mathcal{E}_{j}  \to p_1^*E_{j+1} \to 0$ (see (\ref{NE21})) by $\mathcal{E}_j^* \otimes  p_2^*\mathcal{O}(-1)$, we get
\[
0 \to \mathcal{E}_{j-1} \otimes \mathcal{E}_j^* \to (\text{Ad}(\mathcal{E}_j) \oplus \mathcal{O}) \otimes p_2^*\mathcal{O}(-1) \to \mathcal{E}_j^* \otimes p_1^*E_{j+1} \otimes p_2^*\mathcal{O}(-1) \to 0
\]
By $(a)$, $H^0(\mathcal{E}_j^* \otimes p_1^*E_{j+1} \otimes p_2^*\mathcal{O}(-1))= 0$. Thus it is enough to show that $H^0(\mathcal{E}_j^* \otimes \mathcal{E}_{j-1})= 0$.

Note that by $(d)$, a nilpotent  Higgs fields on $E^j$ which induces the filtration in \eqref{Mar31} on $E^j$,  gives a section $\psi_j \in H^0(\text{Ad}(\mathcal{E}_j \otimes p_1^*K_C \otimes p_2^*\mathcal{O}(-1))$ such that $\psi_j^{j+1}=0$ and vice versa where $j$ is the minimal integer such that, $\psi^{j+1}=0$.  Let $\psi_{j-1} $ be the section of     $ \text{Ad}(\mathcal{E}_{j-1}) \otimes p_1^*K_C \otimes p_2^*\mathcal{O}(-1)$ corresponding to the nilpotent Higgs field $\varphi$ restricted to $E^{j-1}$.

If there is a non-zero morphism $\mathcal{E}_j \to \mathcal{E}_{j-1}$, then the composition with 
\[
\mathcal{E}_{j-1} \xlongrightarrow{\psi_{j-1} }\mathcal{E}_{j-1} \otimes p_1^*K_C \otimes p_2^*\mathcal{O}(-1)) \to \mathcal{E}_j \otimes p_1^*K_C \otimes p_2^*\mathcal{O}(-2) \to  \mathcal{E}_j \otimes p_1^*K_C \otimes p_2^*\mathcal{O}(-1) 
\]
gives a section $\eta$ of $\text{Ad}(\mathcal{E}_j) \otimes p_1^*K_C \otimes p_2^*\mathcal{O}(-1)$, with $\eta^{j}=0$, a contradiction to the minimality of $j$. Thus, $H^0(\mathcal{E}_j^* \otimes \mathcal{E}_{j-1})= 0$.
\end{proof}

\begin{align*}
&\text{Assume that in the family of rational curves through $E$ as constructed in (\ref{NE3}),} \tag{\textasteriskcentered \textasteriskcentered} \label{assumption}\\
&\text{there is a rational curve through $E$ contained in the stable locus.}
\end{align*}

\vspace{0.2 cm}

\begin{remark}\label{R31}
\begin{enumerate}[(i)]
\item If the degree and rank of $E$ are co-primes then every semistable bundle is stable. Thus the hypothesis $($\ref{assumption}$)$ satisfied.

\item Let $E$ be a stable vector bundle of rank $2$ and degree zero. Suppose $E$ can be written as an extension 
\begin{equation}\label{April19}
0 \to \xi \to E \to \xi^{-1} \to 0.
\end{equation}
Let $d:=\text{deg}(\xi)$. Then $d \le -1$. If an extension in $\text{Ext}^1(\xi^{-1}, \xi^)$ is not stable, then it contains  a line subbundle $\eta$ of degree at least zero, which gives a non-zero section of $\eta^{-1} \otimes \xi^{-1}$. Now the degree of $\eta^{-1} \otimes \xi^{-1}$ is $\le -d$. Thus, the dimension of degree zero line bundles $\eta$ such that $\eta^{-1} \otimes \xi^{-1}$ has a section, is $\le -d$. Hence, the dimension of vector bundles containing a degree zero line subbundle $\eta$ such that $\eta^{-1} \otimes \xi^{-1}$ admits a nonzero section, is $\le -d + h^1(\eta^2)-1=g-2-d $.\\
On the other hand, the dimension of the vector bundles which fit in \ref{April19}, is $g-2d$. Therefore, the non-stable bundles occurring in the extension \ref{April19} has co-dimension at least $g-2d-g+2+d= 2-d \ge 
 3$. Thus, there is a rational curve in $\mathbb{P}(\text{Ext}^1(\xi^{-1}, \xi))$ which contained in the stable locus. I hope the assumption is true in general, may be known to the experts. Unfortunately, we are unable to include a proof of it. 
\end{enumerate}
\end{remark}

Let $E \in \mathcal{M}_C^s(n, \delta)$ be a wobbly bundle satisfying (\ref{assumption}) and let $f: \mathbb{P}^1 \longrightarrow \mathcal{M}_C^s(n, \delta)$ be a rational curve as obtained in (\ref{NE3}). Write $\mathcal{E} := \mathcal{E}_m$ (see (\ref{NE21})) and suppose $q$ is the anticanonical degree of the rational curve $f: \mathbb{P}^1 \longrightarrow \mathcal{M}_C^s(n, \delta)$. Then we have the following

\begin{proposition}\label{April8}
$h^0(\mathbb{P}^1, R^1{p_2}_*\text{Ad}(\mathcal{E}) \otimes \mathcal{O}(-1))= q + h^1(C, E_m^* \otimes E_{m+1})$, where $p_2$ is the projection to the second factor and $R^1{p_2}_*\text{Ad}(\mathcal{E})$ denotes the first direct image of $\text{Ad}(\mathcal{E})$.
\end{proposition}
\begin{proof}
Since $E$ is wobbly, we can obtain $E$ by successive extension as in (\ref{NL1}). Let $r_i, d_i$ be the rank and degree of $E_i$, respectively. Then by \cite[Lemma 3.6]{TB}, we have $q = 2\sum_{i < j}(r_id_j-r_jd_i)(j-i)$.

Tensoring the exact sequence \eqref{NE21} by $\mathcal{E}_j^* \otimes p_2^*\mathcal{O}(-1)$, we have
\begin{equation*}\label{NE4}
0 \to \mathcal{E}_{j-1} \otimes \mathcal{E}_j^* \to (\text{Ad}(\mathcal{E}_j)\oplus \mathcal{O}) \otimes p_2^*\mathcal{O}(-1) \to \mathcal{E}_j^* \otimes p_1^*E_{j+1} \otimes p_2^*\mathcal{O}(-1) \to 0
\end{equation*}
Since $H^0(\mathcal{E}_j^* \otimes p_1^*E_{j+1} \otimes p_2^*\mathcal{O}(-1)) = 0$ by Lemma \ref{nnl1} $(a)$, we have
\begin{align}\label{NE7}
\begin{split}
&h^1(\text{ Ad}(\mathcal{E}_j) \otimes p_2^*{\mathcal O}(-1)) = h^1(\mathcal{E}_{j-1} \otimes \mathcal{E}_j^*) + h^1( \mathcal{E}_j^* \otimes p_1^*E_{j+1} \otimes p_2^*\mathcal{O}(-1))\\
&- h^2(\mathcal{E}_{j-1} \otimes \mathcal{E}_j^*) + h^2(\text{Ad}(\mathcal{E}_j) \otimes p_2^*{\mathcal O}(-1))-h^2(\mathcal{E}_j^* \otimes p_1^*E_{j+1} \otimes p_2^*\mathcal{O}(-1))
\end{split}
\end{align}
On the other hand, tensoring the dual sequence of \eqref{NE21} by $\mathcal{E}_{j-1}$, we get
\[
0 \to \mathcal{E}_{j-1} \otimes p_1^*E_{j+1}^*  \to \mathcal{E}_{j-1} \otimes \mathcal{E}_j^*  \to (\text{Ad}(\mathcal{E}_{j-1}) \oplus \mathcal{O})\otimes p_2^*\mathcal{O}(-1) \to 0.
\]
Since $H^0(\text{Ad}(\mathcal{E}_{j-1}) \otimes p_2^*\mathcal{O}(-1)) = 0$ by Lemma \ref{nnl1} $(g)$, from cohomology exact sequence we get 
\begin{align}\label{NE8}
\begin{split}
h^1(\mathcal{E}_{j-1} &\otimes \mathcal{E}_j^*)= h^1(\mathcal{E}_{j-1} \otimes p_1^*E_{j+1}^*) + h^1(\text{Ad}(\mathcal{E}_{j-1})\otimes p_2^*\mathcal{O}(-1))\\
&- h^2(\mathcal{E}_{j-1} \otimes p_1^*E_{j+1}^*) + h^2(\mathcal{E}_{j-1} \otimes \mathcal{E}_j^*)-h^2(\text{Ad}(\mathcal{E}_{j-1})\otimes p_2^*\mathcal{O}(-1)).
\end{split}
\end{align}
Substituting $h^1(\mathcal{E}_j^* \otimes \mathcal{E}_{j-1})$ in \eqref{NE7}, we get
\begin{align}\label{NE9}
\begin{split}
&h^1(\text{ Ad}(\mathcal{E}_j) \otimes p^*{\mathcal O}(-1)) =  h^1(\text{Ad}(\mathcal{E}_{j-1})\otimes p_2^*\mathcal{O}(-1)) + h^1(\mathcal{E}_{j-1} \otimes p_1^*E_{j+1}^*)\\
&+h^1( \mathcal{E}_j^* \otimes p_1^*E_{j+1} \otimes p_2^*\mathcal{O}(-1)) + h^2(\text{ Ad}(\mathcal{E}_j) \otimes p_1^*{\mathcal O}(-1)) - h^2(\mathcal{E}_{j-1} \otimes p_1^*E_{j+1}^*)\\
&- h^2(\text{ Ad}(\mathcal{E}_{j-1}) \otimes p_1^*{\mathcal O}(-1))-h^2(\mathcal{E}_j^* \otimes p_1^*E_{j+1} \otimes p_2^*\mathcal{O}(-1)).
\end{split}
\end{align}

To prove the Lemma, will use induction on $m$. Let $m = 1$. Note that $\mathcal{E}_0 = p_1^*E_1$. By K\"{u}nneth formula, we have
\begin{equation}\label{RE95}
H^1(p_1^*\text{Ad}(E_1) \otimes p_2^*\mathcal{O}(-1))) = H^2(p_1^*\text{Ad}(E_1) \otimes p_2^*\mathcal{O}(-1)) = H^2(p_1^*(E_1 \otimes E_2^*)) = 0.
\end{equation}
Tensoring  the dual exact sequence of \eqref{NE11} by $p_1^*E_1$, we have
\[
0 \to p_1^*(E_1 \otimes E_2^*) \to \mathcal{E}_1^* \otimes p_1^*E_1 \to p_1^*\text{End}(E_1) \otimes p_2^*\mathcal{O}(-1) \to 0
\]
Using \eqref{RE95}, from the cohomology sequence, we get
\begin{equation}\label{RE97}
H^2(\mathcal{E}_1^* \otimes p_1^*E_1) = 0.
\end{equation}
Tensoring  \eqref{NE11} by $\mathcal{E}_1^* \otimes p_2^*\mathcal{O}(-1)$, we have
\[
0 \to \mathcal{E}_1^* \otimes p_1^*E_1  \to (\text{Ad}(\mathcal{E}_1) \oplus \mathcal{O}) \otimes p_2^*\mathcal{O}(-1) \to \mathcal{E}_1^* \otimes p_1^*E_2 \otimes p_2^*\mathcal{O}(-1) \to 0
\]
and using \eqref{RE97} we conclude
\begin{equation}\label{RE96}
H^2(\text{Ad}(\mathcal{E}_1) \otimes p_2^*\mathcal{O}(-1)) = H^2(\mathcal{E}_1^* \otimes p_1^*E_2 \otimes p_2^*\mathcal{O}(-1))
\end{equation}
Taking $j=1$ in \eqref{NE9} and combining \eqref{RE95}, \eqref{RE97} and \eqref{RE96}, we have
\begin{equation}\label{RE94}
h^1(\text{Ad}(\mathcal{E}_1) \otimes p_2^*\mathcal{O}(-1)) = h^1(C, E_1 \otimes E_2^*) + h^1(\mathcal{E}_1^* \otimes p_1^*E_2 \otimes p_2^*\mathcal{O}(-1)).
\end{equation}
Now tensoring  the dual exact sequence of \eqref{NE11} by $p_1^*E_2 \otimes p_2^*\mathcal{O}(-1)$, we have
\[
0 \to p_1^*\text{End}(E_2) \otimes p_2^*\mathcal{O}(-1) \to \mathcal{E}_1^* \otimes p_1^*E_2 \otimes p_2^*\mathcal{O}(-1) \to p_1^*(E_1^* \otimes E_2) \otimes p_2^*\mathcal{O}(-2) \to 0.
\]
By K\"{u}nneth formula, $ H^1(p_1^*\text{End}(E_2) \otimes p_2^*\mathcal{O}(-1)) = 0 = H^2(p_1^*\text{End}(E_2) \otimes p_2^*\mathcal{O}(-1))$. So we get
\[
H^1(\mathcal{E}_1^* \otimes p_1^*E_2 \otimes \mathcal{O}(-1)) = H^0(C, E_1^* \otimes E_2) \otimes H^1(\mathbb{P}^1, \mathcal{O}(-2)) = H^0(C, E_1^* \otimes E_2).
\]
Substituting this in \eqref{RE94}, we get
\[
h^1(\text{Ad}(\mathcal{E}_1) \otimes p_2^*\mathcal{O}(-1))= h^1(C, E_1 \otimes E_2^*) + h^0(C, E_1^* \otimes E_2).
\]
Since $H^1( \text{Ad}(\mathcal{E}_1)) = H^0(R^1p_* \text{Ad}(\mathcal{E}_1))$ by Leray Spectral Sequence, we have
$$
h^0(R^1p_* \text{Ad}(\mathcal{E}_1)) = h^1(C, E_1 \otimes E_2^*) + h^0(C, E_1^* \otimes E_2).
$$
Since $E^1= E$ is stable, $h^0(C, E_1 \otimes E_2^*)=0$. Thus, by Riemann-Roch we have,
\begin{align*}
h^1(C, E_1 \otimes E_2^*) &= (r_1d_2-r_2d_1) \,\, + \,\,  \,\, r_1r_2(g-1) \\
h^0(C, E_1^* \otimes E_2) &= (r_1d_2-r_2d_1) \,\, + \,\, r_1r_2(g-1) + h^0(C, E_1 \otimes E_2^* \otimes K_C).
\end{align*}
Therefore,
\begin{equation}\label{NE10}
h^0(\mathbb{P}^1, R^1{p_2}_*(\text{Ad}(\mathcal{E}_1) \otimes \mathcal{O}(-1)))= 2(r_1d_2-r_2d_1)+ h^0(C, E_1 \otimes E_2^* \otimes K_C),
\end{equation}
which proves our result for $m=1$.

Let's assume $m \geq 2$ and the result holds for $j < m$. Using Lemma \ref{nnl1} $(a), (d)$, from \eqref{NE9}, we have
\begin{align}
\begin{split}
h^1(\text{Ad}(\mathcal{E}_m) &\otimes p_2^*{\mathcal O}(-1)) = h^1(\text{Ad}(\mathcal{E}_{m-1})\otimes p_2^*\mathcal{O}(-1)) + h^1(\mathcal{E}_{m-1} \otimes p_1^*E_{m+1}^*)\\
&+ h^1(\mathcal{E}_m^* \otimes p_1^*E_{m+1} \otimes p_2^*\mathcal{O}(-1)) -h^0(C, E_{m-1} \otimes E_m^* \otimes K_C)\\
&+ h^0(C, E_m \otimes E_{m+1}^* \otimes K_C).
\end{split}
\end{align}
By induction hypothesis 
$$
h^1(\text{Ad}(\mathcal{E}_{m-1})\otimes p_2^*\mathcal{O}(-1)) = 2\sum_{i < j \leq m}(r_id_j-r_jd_i)(j-i) + h^1(C, E_{m-1}^* \otimes E_m).
$$
Write $p^{'}=2\sum_{i < j}(r_id_j-r_jd_i)(j-i), j \le m$.
Using Lemma \ref{nnl1} $(b), (e)$ and $(f)$, we get
\begin{align}
\begin{split}
h^1(\text{Ad}(\mathcal{E}_m) \otimes p_2^*{\mathcal O}(-1)) = p^{'} &+ \sum_{i=0}^{m-1}h^0(C, E_{m-i}^* \otimes E_{m+1}) \cdot h^1(\mathbb{P}^1, \mathcal{O}(-2-i))\\
&+ \sum_{i=0}^{m-1} h^1(C, E_{m-i} \otimes E_{m+1}^*)\cdot h^0(\mathbb{P}^1, \mathcal{O}(i)).
\end{split}
\end{align}
Note that by Remark \ref{NR1}, $h^0(C, E_{m-i} \otimes E_{m+1}^*)= 0$ for $i \ge 0$ and $h^1(C, E_{m-i}^* \otimes E_{m+1} )= 0$ for $i \ge 1$. Using $h^1(\mathbb{P}^1, \mathcal{O}(-2-i)) = h^0(\mathbb{P}^1, \mathcal{O}(i))$, we get
\begin{align*}
h^1(\text{Ad}(\mathcal{E}_m) \otimes p_2^*{\mathcal O}(-1)) = p^{'} &+ \sum_{i=0}^{m-1} (i+1)\big( h^0(C, E_{m-i}^* \otimes E_{m+1})\\
&+h^1(C, E_{m-i} \otimes E_{m+1}^*) \big)\\
= 2\sum_{i < j}&(r_id_j-r_jd_i)(j-i) + h^1(C, E_m^* \otimes E_{m+1}).
\end{align*}
This completes the proof.
\end{proof}

\begin{theorem}
 A  vector bundle $E \in \mathcal{M}_C^s(n, \delta)$ satisfying $($\ref{assumption}$)$ is wobbly,  if and only if there is a  nonfree  rational curve $\ell$ in $\mathcal{M}_C^s(n, \delta)$ passing through $E$. In particular, if the rank and degree are co-prime, then $E$ is wobbly, if and only if there is a  nonfree  rational curve $\ell$ in $\mathcal{M}_C^s(n, \delta)$ passing through $E$.
\end{theorem}
\begin{proof}
Let $\ell$ be a nonfree  rational curve in $\mathcal{M}_C^s(n, \delta)$ containing $E$. Let $h: T^*\mathcal{M}_C^s(n, \delta) \to W$ be the restriction of the Hitchin map to the co-tangent bundle as in \eqref{RE98}. Let $h|_{\ell}$  be the restriction of $h$ to ${T^*\mathcal{M}_C^s(n, \delta)}|_{\ell}$. Since $\ell$ is nonfree, ${T^*\mathcal{M}_C^s(n, \delta)}|_{\ell}= \oplus_{i=1}^N\mathcal{O}_{\mathbb{P}^1}(a_i), a_1 \ge a_2\ge ...\ge a_N$ with $a_1 > 0$.

Let $s$ be a nonzero section of ${T^*\mathcal{M}_C^s(n, \delta)}|_{\ell}$ corresponding to the direct summand $\mathcal{O}_{\mathbb{P}^1}(a_1)$. Since $\ell$ is complete, $h(s(\ell))$ is constant. On the other hand, for any section $s \in H^0(\mathbb{P}^1, \mathcal{O}_{\mathbb{P}^1}(a_1))$ there is a point $p \in \ell$ such that $s(p)=0$. Hence $h(s(\ell))= 0$. Since $\mathcal{O}_{\mathbb{P}^1}(a_1)$ is globally generated,  there is a section $s$ such that $s(p) \ne 0$, but $h(s(\ell))=0$.  Thus the restriction of the Hitchin map to the cotangent space $T_p^*\mathcal{M}_C(n, \delta) \to W$ has positive dimensional fiber over zero for every $p \in \ell$. By \cite[Corollary 1.2]{CP}, $p$ is not very stable bundle. In particular, $E$ is wobbly.

Conversely, let $E$ be a wobbly bundle.  We will show that there is a nonfree rational curve in $\mathcal{M}_C^s(n, \delta)$ through $E$. The key idea is as follows:\\
If $f: \mathbb{P}^1 \to \mathcal{M}_C^s(n, \delta)$ be a free rational curve of anti-canonical degree $t$, then $f^*T\mathcal{M}_C^s(n, \delta)= \oplus_{i=1}^N\mathcal{O}_{\mathbb{P}^1}(b_i)$ such that $b_1 \geq \cdots \ge b_N \ge 0$ and $\sum_{i=1}^N b_i= t$. Thus $h^0(\mathbb{P}^1, f^*T\mathcal{M}_C^s(n, \delta)) = N +t$ and $h^0(\mathbb{P}^1, f^*T\mathcal{M}_C^s(n, \delta) \otimes \mathcal{O}(-1))= \sum_{i=1}^Nb_i= t$. On the other hand, if $f: \mathbb{P}^1 \to \mathcal{M}_C^s(n, \delta)$ is not free, then we have $b_N < 0$. Let $u \in \{ 1, \cdots, N \}$ be the largest integer such that $b_u \ge 0$. Then $u < N$, $\sum_{i=1}^{u}b_i = t + \sum_{j=u+1}^N |b_j|$ and
 \begin{equation}\label{RE80}
 h^0(\mathbb{P}^1, f^*T\mathcal{M}_C^s(n, \delta) \otimes \mathcal{O}(-1))=\sum_{i=1}^u((b_i-1)+1)=\sum_{i=1}^ub_i=t + \sum_{j=u+1}^N|b_j|.
 \end{equation}
Since $\sum_{j=u+1}^N|b_j| >1$, $h^0(\mathbb{P}^1, f^*T\mathcal{M}_C^s(n, \delta) \otimes \mathcal{O}(-1)) > t$.  Thus it is sufficient to produce a rational curve  $f:\mathbb{P}^1 \to \mathcal{M}_C^s(n, \delta)$ through $E$ such that $h^0(\mathbb{P}^1, f^*T\mathcal{M}_C^s(n, \delta) \otimes \mathcal{O}(-1)) > t$.

Let $f:\mathbb{P}^1 \to \mathcal{M}_C(n, \delta)$ be a rational curve through $E$ described in \eqref{NE3}. The anticanonical degree of the rational curve is $ t:=2\sum_{i < j}(r_id_j-r_jd_i)(j-i)$. Note that $f^*T\mathcal{M}_C^s(n, \delta)$ can be identified with $R^1{p_2}_*\text{Ad}(\mathcal{E}_m)$. From \eqref{NL1}, we have $h^1(C, E_m^* \otimes E_{m+1}) = h^0(C, E_m \otimes E_{m+1}^* \otimes K_C) > 0$. Now the Theorem follows from Lemma \ref{April8}.
\end{proof}

\begin{remark}\label{R100}
Note that if $h^0(\mathbb{P}^1, f^*T\mathcal{M}_C^s(n, \delta))= t+1$, then from \eqref{RE80}, we have $\sum_{i=u+1}^N|b_i|=1$. Thus, $u= N-1$ and $b_N=-1$. In other words, the restriction of the  tangent bundle $T\mathcal{M}_C^s(n, \delta)$ to the rational curve $f:\mathbb{P}^1 \to \mathcal{M}_C^s(n, \delta)$ has unique direct summand of negative degree.   
\end{remark}

\section{Purity of Wobbly Locus}\label{S4}

Let $\mathcal{W}_C$ be the set of wobbly bundles. By definition, its complement consists of very stable bundles. By \cite[Proposition 3.5]{L}, a very stable vector bundle is stable
and that the locus of very stable bundles is a non-empty open subset of $\mathcal{M}_C(n, \delta)$.
Hence, $\mathcal{W}$ is a closed subset of $\mathcal{M}_C(n, \delta)$.

Let $E \in \mathcal{W}_C$. Then there is a non-zero nilpotent Higgs field $\varphi$ which induces a filtration as in \ref{Mar31}. Thus, $E$ can be written by successive extensions as in \eqref{NL1}. Our goal is to show that if $E$ is a general point of some irreducible component of $\mathcal{W}_C$, then $h^0(C, E_{m+1}^* \otimes E_m \otimes K_C)= 1$. Let $r_m, r_{m+1}, d_m, d_{m+1}$ are respective rank and degree of the vector bundles $E_m$ and $E_{m+1}$ appearing in  \eqref{NL1} and set $s_{(E, \varphi)}:= r_md_{m+1}-r_{m+1}d_m$. Note that if $s_{(E, \varphi)} < r_mr_{m+1}(g-1)-1$, then by Riemann-Roch, $h^0(C, E_{m+1}^* \otimes E_m \otimes K_C) \ge 2$.

Let $\mathcal{W}_1$ be the set of all  wobbly bundles $E$ such that $E$ admits a nonzero nilpotent Higgs field $\varphi$ with $s_{(E, \varphi)} < r_mr_{m+1}(g-1)-1$ and let $\mathcal{W}_2 := \mathcal{W}_C \setminus \mathcal{W}_1$.

\begin{proposition}\label{NP1}
$\mathcal{W}_1 \subset \overline{\mathcal{W}_2}$, where $\overline{\mathcal{W}_2}$ denotes the closure of $\mathcal{W}_2$.
\end{proposition}
\begin{proof}
Our idea is to use \cite[Proposition  1.9]{RTLC} to deform the elements in $\mathcal{W}_1$ to an element in $\mathcal{W}_2$. Let $E \in \mathcal{W}_1$.  We have short exact sequences
  \[
0 \to E^{m-1} \to E \to E_{m+1} \to 0 \,\,\, \text{and} \,\,\, 0 \to E^{m-2} \to E^{m-1} \to E_m \to 0
\] 
such that $h^0(C,  E_{m+1}^* \otimes E_m \otimes K_C) \ge 1$ and $s_{(E, \varphi)} < r_mr_{m+1}(g-1)$. Then
as in \cite[Proposition 1.9]{RTLC}, one can deform $E$ to a stable vector bundle $\widetilde{E}$ such that $\widetilde{E}$ fits in an exact sequence of the form
\[
 0 \to \widetilde{E}^{m-1} \to \widetilde{E} \to \widetilde{E}_{m+1} \to 0
 \] 
 and 
 $\widetilde{E}^{m-1}$ fits in an exact sequence of the form
  \[
  0 \to \widetilde{E}^{m-2} \to \widetilde{E}^{m-1} \to \widetilde{E}_m \to 0,
  \]
where rank $\widetilde{E} =$ rank $E$, $\deg \widetilde{E} = \deg E$ and $\deg \widetilde{E}^{m-1} = \deg E^{m-1} - 1$. The degrees of $\widetilde{E}_m$ and $\widetilde{E}_{m+1}$ are one more than that of   $E_{m+1}$ and $E_{m}$, respectively, and $h^0(C, {\widetilde{E}_{m+1}}^* \otimes \widetilde{E}_m \otimes K_C) \ne 0$.

For convenience, we will quickly recall the construction of $\widetilde{E}$. Consider the following diagrams of elementary transformations:
\begin{equation}
\begin{tikzpicture}[baseline=(current  bounding  box.center)][font=\small]
  \matrix(m)[matrix of math nodes,row sep=2.0em, column sep=4.5em,text height=1.5ex, 
 text depth=0.25ex]
 { &  0 & 0 \\
 0 & \widetilde{E}^{m-1} & \overline{E} & E_{m+1} & 0  \\
 0 & E^{m-1} & E & E_{m+1} & 0 \\
&  \mathcal{O}_P & \mathcal{O}_P \\
&  0 & 0  \\};
\path[->](m-1-2) edge (m-2-2);
\path[->](m-2-2) edge (m-3-2);
\path[->](m-3-2) edge (m-4-2);
\path[->](m-4-2) edge (m-5-2);
\path[->](m-1-3) edge (m-2-3);
\path[->](m-2-3) edge (m-3-3);
\path[->](m-3-3) edge (m-4-3);
\path[->](m-4-3) edge (m-5-3);
\path[->](m-2-1) edge (m-2-2);
\path[->](m-2-2) edge  (m-2-3);
\path[->](m-2-3) edge  (m-2-4);
\path[->](m-2-4) edge  (m-2-5);
\path[->](m-3-1) edge (m-3-2);
\path[->](m-3-2) edge  (m-3-3);
\path[->](m-3-3) edge  (m-3-4);
\path[->](m-3-4) edge  (m-3-5);
\path[->] (m-2-4) edge (m-3-4);
 \end{tikzpicture}
\end{equation}

\begin{equation}\label{FEB1}
\begin{tikzpicture}[baseline=(current  bounding  box.center)][font=\small]
  \matrix(m)[matrix of math nodes,row sep=2.0em, column sep=4.5em,text height=1.5ex, 
 text depth=0.25ex]
 { &  0 & 0 \\
 0 & \widetilde{E}_{m+1}^*  & \widetilde{E}^* & {(\widetilde{E}^{m-1})}^* & 0  \\
 0 & E_{m+1}^*  & \overline{E}^* & {(\widetilde{E}^{m-1})}^* & 0   \\
&  \mathcal{O}_P & \mathcal{O}_P \\
&  0 & 0  \\};
\path[->](m-1-2) edge (m-2-2);
\path[->](m-2-2) edge (m-3-2);
\path[->](m-3-2) edge (m-4-2);
\path[->](m-4-2) edge (m-5-2);
\path[->](m-1-3) edge (m-2-3);
\path[->](m-2-3) edge (m-3-3);
\path[->](m-3-3) edge (m-4-3);
\path[->](m-4-3) edge (m-5-3);
\path[->](m-2-1) edge (m-2-2);
\path[->](m-2-2) edge  (m-2-3);
\path[->](m-2-3) edge  (m-2-4);
\path[->](m-2-4) edge  (m-2-5);
\path[->](m-3-1) edge (m-3-2);
\path[->](m-3-2) edge  (m-3-3);
\path[->](m-3-3) edge  (m-3-4);
\path[->](m-3-4) edge  (m-3-5);
\path[->] (m-2-4) edge (m-3-4);
 \end{tikzpicture}
\end{equation}
Taking dual of the first horizontal exact sequence in \eqref{FEB1}, we have 
\[
 0 \to \widetilde{E}^{m-1} \to \widetilde{E} \to \widetilde{E}_{m+1} \to 0. \,\, \,
 \]
By similar construction we have, 
  \[
  0 \to \widetilde{E}^{m-2} \to \widetilde{E}^{m-1} \to \widetilde{E}_m \to 0.
  \]
Since $h^0(C, E_{m+1}^* \otimes E_m \otimes K_C) \ne 0$, a nonzero morphism $E_{m+1} \to E_m \otimes K_C$ induces a morphism $\widetilde{E}_{m+1} \to \widetilde{E}_m \otimes K_C$ with the following  commutative diagram:

\begin{equation}\label{FEB2}
\begin{tikzpicture}[baseline=(current  bounding  box.center)][font=\small]
  \matrix(m)[matrix of math nodes,row sep=2.0em, column sep=4.0em,text height=1.5ex, 
 text depth=0.25ex]
 { 0 & E_{m+1}  & \widetilde{E}_{m+1} & \mathcal{O}_P & 0  \\
 0 & E_m \otimes K_C  & \widetilde{E}_m \otimes K_C & \mathcal{O}_P \otimes K_C & 0   \\};
\path[->](m-1-1) edge (m-1-2);
\path[->](m-1-2) edge (m-1-3);
\path[->](m-1-3) edge (m-1-4);
\path[->](m-1-4) edge (m-1-5);
\path[->](m-1-2) edge (m-2-2);
\path[->] (m-1-3) edge (m-2-3);
\path[->](m-1-4) edge (m-2-4);
\path[->](m-2-1) edge (m-2-2);
\path[->](m-2-2) edge  (m-2-3);
\path[->](m-2-3) edge  (m-2-4);
\path[->](m-2-4) edge  (m-2-5);
\end{tikzpicture}
\end{equation}
The stability of a general vector bundle $\widetilde{E}$ obtained this way follows from \cite[Proposition  1.9]{RTLC}.  Note that the bundle $\widetilde{E}$ admits a non-zero Higgs fild $\widetilde{\varphi}$ such that  $s_{(\widetilde{E}, \widetilde{\varphi})}= s_{(E, \varphi)}+ r_m + r_{m+1}$. If $s_{E}+ r_m + r_{m+1} < r_mr_{m+1}(g-1)$, then we continue the process to deform $E$ to a stable bundle $E^{'}$  with a nilpotent Higgs field $\varphi'$ such that $E^{'}$ can be obtained as  extensions 
  \[
0 \to {E'}^{m-1} \to E' \to {E'}_{m+1} \to 0 \,\,\, \text{and} \,\,\, 0 \to {E'}^{m-2} \to {E'}^{m-1} \to {E'}_m \to 0
\] 
with $H^0(C, ({E'}_{m+1})^* \otimes {E'}_m \otimes K_C) \ne 0$ and $s_{(E', \varphi')} \ge r_mr_{m+1}(g-1)$.
  Thus, $\mathcal{W}_1 \subset \overline{\mathcal{W}_2}$. 
\end{proof}


\begin{lemma}\label{NL2}
Let $\mathcal{W}_2^1$ be any component of  $\mathcal{W}_2$ and let $E \in \mathcal{W}_2^1$ be a general point satisfying the hypothesis \eqref{assumption}. Then there is a rational curve $f:\mathbb{P}^1 \to \mathcal{M}_C^s(n, \delta)$ passing through $E$ such that $f^*T\mathcal{M}_C^s(n, \delta)$ has exactly one direct summand of negative degree.
\end{lemma}
\begin{proof}
We can obtain $E$ by successive extensions as in \eqref{NL1} with $h^0(C, {E_{m+1}}^* \otimes E_m \otimes K_C) \ne 0 $ and $s_{E}\ge r_mr_{m+1}(g-1)$. Let ${E'}_m$ be a generic vector bundle of rank $r_m$ and degree $d_m$ such that $r_md_{m+1}-r_{m+1}d_m \ge r_mr_{m+1}(g-1)$. Let $W^0_{r_{m+1}, d_{m+1}}({E'}_m \otimes K_C) $ be the twisted Brill-Noether loci of stable bundles ${E'}_{m+1}$ of rank $r_{m+1}$ and degree $d_{m+1}$ such that $h^0(C, {{E'}_{m+1}}^* \otimes {E'}_m \otimes K_C) \ge 1$. Then the expected dimension $\rho^0$ of the loci is given by
\begin{equation}
\begin{split}
  \rho^0 & = r_{m+1}(r_{m+1}-r_m)(g-1) +r_{m+1}(r_m(2g-2)+d_m)-r_md_{m+1}\\
& ={r_{m+1}}^2(g-1)+r_mr_{m+1}(g-1) -(r_md_{m+1}-r_{m+1}d_m) 
 \end{split}
 \end{equation}
 
By \cite[Theorem 0.3]{RTLC} and its proof, if $(r_md_{m+1}-r_{m+1}d_m)-r_mr_{m+1}(g-1) <{r_{m+1}}^2(g-1)$, then  $W^0_{r_{m+1}, d_{m+1}}({E'}_m \otimes K_C) $ is irreducible of the expected dimension $\rho^0$. In fact for irreducibility one can argue as follows: If $F \in W^0_{r_{m+1}, d_{m+1}}({E'}_m \otimes K_C)$ then the map $F \to {E'}_m \otimes K_C$ allows us to write $F$ as an extension of a vector bundle $F_1$ of smaller rank and certain  degree by a subbundle $F_2 \subset {E'}_m \otimes K_C$ of complementary rank and degree. Now consider the spaces of such extensions. The moduli spaces of  semistable vector  bundles of fixed rank  is irreducible. The fibers are $H^1(F_2^*\otimes F_1)$, the stability of the extension guarantees that these dimensions are constant. It follows then that the total space is irreducible.

Note that if ${E'}_{m+1}$ is a general element of $W^0_{r_{m+1}, d_{m+1}}({E'}_m \otimes K_C)$ then one has 
\[
h^0(C, ({E'}_{m+1})^* \otimes {E'}_m \otimes K_C)=1.
\]
Thus, if $E \in \mathcal{W}_2^1$ is general, then $h^0(C, {E_{m+1}}^* \otimes E_m \otimes K_C)=1$. Therefore, by remark \eqref{R100}, the rational curve through $E$ constructed in previous section is non-free and the restriction of the tangent bundle to it has exactly one direct summand of negative degree which is $-1$.
\end{proof}

\begin{proposition}\label{NP2}
Assume that a general element of every component of $\mathcal{W}_2$ satisfies the hypothesis \eqref{assumption}. Let  $E$ be a general point of an irreducible component of $\mathcal{W}_2$. Then  there is an irreducible component of $\mathcal{W}$ containing $E$ of co-dimension one in $\mathcal{M}_C^s(n, \delta)$.
\end{proposition}
\begin{proof}
Let $E$ be a general element of an irreducible component $\mathcal{W}_2^1$ of $\mathcal{W}_2$. Then by Lemma \eqref{NL2}, there is a rational curve $f: \mathbb{P}^1 \to \mathcal{M}_C^s(n, \delta)$ through $E$ such that $f^*T\mathcal{M}_C^s(n, \delta)$ has exactly one direct summand of negative degree which is $-1$. Let $\text{Hom}^{\text{nonfree}}(\mathbb{P}^1, \mathcal{M}_C^s(n, \delta))$ be the closed subset of $\text{Hom}(\mathbb{P}^1, \mathcal{M}_C^s(n, \delta))$ consisting of nonfree rational curves in $\mathcal{M}_C^s(n, \delta)$ and let $V$ be an irreducible component of $\text{Hom}^{\text{nonfree}}(\mathbb{P}^1, \mathcal{M}_C^s(n, \delta))$ containing $f$. Then for a general element $g:\mathbb{P}^1 \to \mathcal{M}_C^s(n, \delta)$ in $V$, $g^*T\mathcal{M}_C^s(n, \delta)$ has exactly one direct summand of negative degree. 
Therefore,  by Lemma \ref{July7} $\overline{F(\mathbb{P}^1 \times V)}$ has co-dimension one in $\mathcal{M}_C^s(n, \delta)$. Note that every element in $V$ is nonfree. Since $f \in V, E \in \overline{F(\mathbb{P}^1 \times V)} $, $\overline{F(\mathbb{P}^1 \times V)} \subset \mathcal{W}$, which proves the proposition.
\end{proof}

 \begin{theorem}
 If $(n, d)=1$, then $\mathcal{W}_C$ is of pure codimension one.
 \end{theorem}
 \begin{proof}
 By remark \ref{R31}, every vector bundle $E \in \mathcal{W}_C$, satisfies the hypothesis \eqref{assumption}. Then the Theorem follows from the Proposition \ref{NP1} and the Proposition \ref{NP2}.
 \end{proof}

\vspace{5 mm}

{\it Acknowledgement:} I would like to thank to Dr. Suratno Basu and Dr. Krishanu Dan for many useful discussions, comments and suggestions which helped me a lot to make the exposition better. I also would like to thank Prof.  Montserrat Teixidor I Bigas for many important discussions. I also would like thank Prof. Izzet Coskun and Prof. Eric Riedl  for Lemma \ref{July7}.

\end{document}